\documentclass[11pt, a4paper]{amsart}
\usepackage{mathrsfs,amsthm,amssymb,amsmath,url,cite}
\newcommand{\acts}{\curvearrowright}

\theoremstyle{plain}
    \newtheorem{thm}{Theorem}
    
    \newtheorem{lem}[thm]{Lemma}
    \newtheorem{prop}{Proposition}
    \newtheorem{cor}[thm]{Corollary}
    
    \newtheorem{quest}[thm]{Question}

\theoremstyle{definition}
    \newtheorem{ex}[thm]{Example}
    \newtheorem{defn}[thm]{Definition}

\newcommand{\fG}{{\bf G}}
\newcommand{\fH}{{\bf H}}

\DeclareMathOperator{\Sym}{Sym}

\newcommand{\rest}{{\upharpoonright}}

\DeclareMathOperator{\id}{id}\DeclareMathOperator{\Aut}{Aut}

\newcommand{\G}{{\mathscr G}}

\newcommand{\To}{\rightarrow}

\newcommand{\Hung}{\H}
\renewcommand{\G}{{\mathbf G}}
\renewcommand{\fH}{{\mathbf H}}

\title[Canonical Functions: a Proof via topological dynamics]{Canonical Functions:\\ a Proof via topological dynamics}
\author%[M. Bodirsky]
{Manuel Bodirsky}
    \address{Institut f\"{u}r Algebra\\TU Dresden\\01062 Dresden\\Germany}
    \email{Manuel.Bodirsky@tu-dresden.de}
   \urladdr{http://www.math.tu-dresden.de/~bodirsky/}
\author%[M.~Pinsker]
{Michael Pinsker}
	\address{Institute of Discrete Mathematics and Geometry, Technische Universit\"at Wien, Austria, and Department of Algebra, Charles University in Prague, Czech Republic}        \email{marula@gmx.at}
    \urladdr{http://dmg.tuwien.ac.at/pinsker/}
    \thanks{The first author has received funding from the European Research Council (Grant Agreement no. 681988, CSP-Infinity). The second author has received funding from project  P27600 of the  Austrian Science Fund (FWF)}

\date{\today}

\begin{document}
\maketitle

\begin{abstract}
Canonical functions are a powerful concept with numerous applications in the study of groups, monoids, and clones on countable structures with Ramsey-type  properties. In this short note, we present a proof of the existence of canonical functions in certain sets using topological dynamics, providing a shorter alternative to the original combinatorial argument. We moreover present equivalent algebraic characterisations of canonicity. 
\end{abstract}

\section{Introduction}
When $f\colon (\mathbb Q;<)\To(\mathbb Q;<)$ is any function from the order of the rational numbers to itself, then there are arbitrarily large finite subsets of $\mathbb Q$ on which $f$ ``behaves regularly"; that is, it is either strictly increasing, strictly decreasing, or constant. A direct (although arguably unnecessarily elaborate) way to see this is by applying Ramsey's theorem: two-element subsets of $\mathbb Q$ are colored with three colors according to the local behavior of $f$ on them (this yields, by the infinite version of Ramsey's theorem, even an infinite set on which $f$ behaves regularly, but this is beside the point for us). In particular, it follows that the closure of the set $\{\beta\, f \, \alpha\; |\; \alpha,\beta\in\Aut(\mathbb Q;<)\}$ in ${\mathbb Q}^{\mathbb Q}$, equipped with the pointwise convergence topology, contains a function which behaves regularly everywhere. This function of regular behavior is called \emph{canonical}.

More generally, a function $f\colon \Delta\To \Lambda$ between two structures $\Delta, \Lambda$ is called canonical if it behaves regularly in an analogous way, that is, if it sends tuples in $\Delta$ of the same type (in the sense of model theory, as in~\cite{Hodges}) to tuples the same type in $\Lambda$~\cite{BPT-decidability-of-definability, RandomMinOps, BP-reductsRamsey}. Similarly as in the example above, canonical functions can be obtained from $f$, in the fashion stated above, if $\Delta$ has sufficient Ramsey-theoretic properties (for example, if it is a countable Ramsey structure in the sense of~\cite{Topo-Dynamics}) and if $\Lambda$ is sufficiently small (for example countable and $\omega$-categorical)~\cite{BPT-decidability-of-definability,RandomMinOps,BP-reductsRamsey}. 

The concept of canonical functions has turned out useful in numerous applications: for classifying first-order reducts 
they are used in~\cite{agarwal, PongraczHensonReducts,Poset-Reducts,42,BodJonsPham,AgarwalKompatscher,LinmanPinsker,RandomMinOps}, 
for computing the model companion and the model-complete core of such reducts~\cite{MottetPinskerCores}, 
for complexity classification for constraint satisfaction problems (CSPs) in~\cite{BMPP16,equiv-csps,BodPin-Schaefer-both,Phylo-Complexity,KompatscherPham}, 
for decidability of meta-problems in the context of the CSPs in~\cite{BPT-decidability-of-definability}, 
for lifting algorithmic results from finite-domain CSPs to CSPs over infinite domains in~\cite{Bodirsky-Mottet}, 
for lifting algorithmic results from finite-domain CSPs to homomorphism problems from definable infinite structures 
to finite structures~\cite{locally-finite}, 
and  
for decidability questions in computations with atoms in~\cite{definable-homomorphisms}. Most of these applications are covered by a survey article published shortly after their invention~\cite{BP-reductsRamsey}. 

As indicated above, the technique is available for a function $f\colon \Delta\To\Lambda$, in particular, whenever $\Delta$ is a countable Ramsey structure and $\Lambda$ is countable and $\omega$-categorical, and the existence of canonical functions in the set $\overline{\{\beta\, f\, \alpha\; |\; \alpha\in \Aut(\Delta),\beta\in\Aut(\Lambda)\}}\subseteq \Lambda^\Delta$
was originally shown under these conditions by a combinatorial argument~\cite{BPT-decidability-of-definability,RandomMinOps,BP-reductsRamsey}. 
By the Kechris-Pestov-Todor\v{c}evi\'{c} correspondence~\cite{Topo-Dynamics},
a countable structure $\Delta$ is Ramsey (with respect to colorings of embeddings) if and only if its automorphism group $\Aut(\Delta)$ is \emph{extremely amenable}, meaning that every continuous action of it on a compact Hausdorff space has a fixed point. Moreover, by the theorem of Ryll-Nardzewski, Engeler, and Svenonius, two tuples in a countable $\omega$-categorical structure have the same type if and only if they lie in the same orbit with respect to the componentwise action of its automorphism group on tuples, and a countable structure is $\omega$-categorical if and only if its automorphism group is oligomorphic. Therefore both the definition of canonicity as well as the above-mentioned conditions implying their existence in sets of the form $\overline{\{\beta\, f\, \alpha\; |\; \alpha\in \Aut(\Delta),\beta\in\Aut(\Lambda)\}}$ can be formulated in the language of permutation groups.

It is therefore natural to ask for a perhaps more elegant proof of the existence of canonical functions via topological dynamics, reminiscent of the numerous proofs of combinatorial statements obtained in a similar fashion (cf.~the survey~\cite{bergelson} for Ergodic Ramsey theory;  \cite{Kechris-nonarchimedian} mentions some applications of extreme amenability). In this short note, 
we present such a proof. The proof was discovered by the authors at the Workshop on Algebra and CSPs at the Fields Institute in Toronto in 2011, where it was also presented (by the second author), 
 but has so far not appeared in print. 
We use the occasion of this note to present various equivalent characterisations of canonicity of functions that facilitate their use and
better explain their significance. 

\section{Canonicity}
\label{sect:prelims}
We use the notation $\fG \acts X$
to denote a permutation group $\fG$ acting on a set $X$. 
We make the convention that if $f \colon X \to Y$ is a function and $t = (t_1,\dots,t_k) \in X^k$, where $k\geq 1$, 
then $f(t) := (f(t_1),\dots,f(t_k)) \in Y^k$ denotes the $k$-tuple obtained by
applying $f$ to $t$ componentwise.

The following is an algebraic formulation of Definition~6 in~\cite{BPT-decidability-of-definability}.

\begin{defn}
Let $\fG \acts X$ and $\fH \acts Y$ be
permutation groups. 
A function $f \colon X \to Y$ is
called \emph{canonical with respect to $\fG$ and $\fH$} 
if for every finite tuple $t \in X^{<\omega}$ and every $\alpha \in \fG$
there exists $\beta \in \fH$ such that
$f\,\alpha(t) = \beta\, f(t)$. 
\end{defn}

Hence, functions that are canonical
with respect to $\fG$ and $\fH$
induce for each integer $k\geq 1$ a function from
the orbits of the componentwise action 
of $\fG$ of $X^k$
to the orbits of the componentwise action of 
$\fH$ on $Y^k$.

In order to formulate properties equivalent to canonicity we require some topological notions. 
We consider the set $Y^X$ of all functions 
from $X$ to $Y$ as a topological space equipped with the topology of pointwise convergence, i.e., the product topology where $Y$ is taken to be discrete. 
When $S\subseteq Y^X$, then we write $\overline S$ for the closure of $S$ in this space. In particular, when $\G \acts X$ is a permutation group, then $\overline \fG$ is the closure of $\G$ in $X^X$. Note
that $\overline \fG$ might no longer be a group, but it is still a monoid with respect to composition of functions.
For example, in the case of the full symmetric
group $\fG = \Sym(X)$ consisting of all permutations of $X$,  
 $\overline \fG$ is the transformation monoid of all injections in $X^X$. 

A permutation group ${\bf G} \acts X$ is called
\emph{oligomorphic} if for each $k \geq 1$
the componentwise action of $\bf G$ on $X^k$
has finitely many orbits. 
For oligomorphic permutation groups
we have the following equivalent characterisations of canonicity.

\begin{prop}\label{prop:canonical}
Let $\fG \acts X$ and $\fH \acts Y$ be permutation groups, where $X, Y$ are countable and $\fH \acts Y$ is oligomorphic.
Then for any function $f \colon X \to Y$
the following are equivalent.
\begin{enumerate}
\item $f$ is canonical with respect to $\fG$ and $\fH$;
\item for all $\alpha \in {\bf G}$ we
have $f \alpha \in \overline{\fH f}:= \overline{ \{ \beta f \mid \beta \in \fH\} }$;
\item for all $\alpha \in \fG$ 
there are $e_1,e_2 \in \overline{\fH}$ such that
$e_1 f \alpha = e_2 f$. 
\end{enumerate}
\end{prop}

A stronger condition would be to require that 
for all $\alpha \in \fG$ 
there is an $e \in \overline{\fH}$ such that
$f \alpha = e f$. 
To illustrate that this is strictly stronger, already
when $\fG = \fH$, 
we give an explicit example.

\begin{ex}[thanks to Trung Van Pham]\label{ex:trung}
Let $\fG:=\fH := \Aut({\mathbb Q};<)$. Note that $({\mathbb Q};<)$
and $({\mathbb Q} \setminus \{0\};<)$ are isomorphic, and let $f$ be such an isomorphism. Then $f$, viewed as a function 
 from ${\mathbb Q} \to {\mathbb Q}$, is clearly canonical with respect to $\fG$ and $\fH$. 
But $f$ does not satisfy the stronger condition above.
%:there is no $e \in \overline{\fG}$ such that $f \alpha = e f$. 
To see this, choose
$a \in {\mathbb Q}$ such that $f(a)<0$, and pick $\alpha \in \fG$ such that $f\alpha(a)>0$. Since 
the image of $f \alpha$ equals the image of $f$, 
any $e \in \overline{\fH}$ such that 
$f \alpha = e f$ must fix $0$. 
Since $e$ must also preserve $<$, it cannot map
$f(a) < 0$ to 
$f \alpha(a) > 0$. Hence, there is no $e \in \overline{\fH}$ such that
$f \alpha = e f$.  \qed
\end{ex}

% In general, we have two lift lemmas:
% one from going from pp sentences that are satisfied
% locally to operations that satisfy the equations locally.
% The second step is to find for functions that 
% satisfy equations locally, some units such that 
% those functions satisfy an equation globally. 
% The second type of lift is needed here. 

%\section{A lift lemma}

In Proposition~\ref{prop:canonical}, the
implications from (1) to (2) and from (3) to (1)
follow straightforwardly from
the definitions. For the implication from (2) to (3)
we need a lift lemma, which is in essence from~\cite{BPP-projective-homomorphisms}. 
This lemma has been applied frequently
lately~\cite{Phylo-Complexity,BartoPinskerDichotomy,Topo, Bodirsky-Mottet}, in various slightly different forms. We need yet another formulation here; since the lemma is a consequence of a compactness argument which we need in any case for the canonisation theorem in Section~\ref{sect:canonisation}, we present its proof.

Let $\fH\acts Y$ be a permutation group, and let $f,g\in Y^X$, for some $X$. We say that \emph{$f=g$ holds locally modulo $\fH$} if for all finite $F\subseteq X$ there exist $\beta_1,\beta_2 \in \fH$ such that $\beta_1\, f\rest_F=\beta_2\, g\rest_F$. We say that \emph{$f=g$ holds globally modulo $\fH$ (modulo $\overline{\fH}$)} if there exist $e_1,e_2\in \fH$ ($e_1,e_2\in\overline{\fH}$, respectively) such that $e_1\, f=e_2\, g$.

Of course, if $f=g$ holds globally modulo $\overline{\fH}$, then it holds locally modulo $\fH$. On the other hand, if $f=g$ holds locally modulo $\fH$, then it need not hold globally modulo $\fH$: an example are the functions $f$ and $f\alpha$ in Example~\ref{ex:trung}, for the reasons explained above. 
However, there exist $e_1,e_2 \in \overline{\fH}$ such that
$e_1\,f = e_2\, f\, \alpha$, so $f=f\alpha$ holds globally modulo $\overline{\fH}$. This is true in general, as we see in the following lift lemma.

\begin{lem}\label{lem:lift}
Let $\fH\acts Y$ be an oligomorphic permutation group acting on a countable set $Y$, let $I$ be a countable index set, and let $X_i$ be a countable set for every $i\in I$. 
 Let $f_i,g_i$ be functions in $Y^{X_i}$ such that $f_i=g_i$ holds locally modulo $\fH$ for all $i\in I$. Then $f_i=g_i$ holds globally modulo $\overline{\fH}$ for all $i\in I$, and in fact 
  there exist $e, e_i \in {\overline \fH}$ such that $e\, f_i=e_i\, g_i$ for all $i\in I$.
\end{lem}

To prove Lemma~\ref{lem:lift}, it is convenient to work with a certain compact Hausdorff space that we also use for the canonisation theorem in Section~\ref{sect:canonisation}. 
Let $\fH\acts Y$ be a permutation group, and $X$ be a set. On $Y^X$, define an equivalence relation $\sim$ by setting $f\sim g$ if $f\in \overline{\fH\, g}$, i.e., if $f=g$ holds locally modulo $\fH$; here, transitivity and symmetry follow from the fact that $\fH$ is a group. 
The following has essentially been shown in~\cite{Topo-Birk} (though for the finer equivalence relation of global equality modulo $\fH$), but we give an argument for the convenience of the reader since it is used so often (cf.~for example~\cite{BodJunker,wonderland, BPP-projective-homomorphisms, BKOPP,BKOPP-equations}).

\begin{lem}%[\cite{Topo-Birk}]
\label{lem:space}
If $\fH\acts Y$ is oligomorphic, and $X$ is countable, then the space $Y^X/_\sim$ is a compact Hausdorff space.
\end{lem}
\begin{proof}
We represent the space in such a way that this becomes obvious. Extend the definition of the equivalence relation $\sim$ to all spaces $Y^F$, where $F\subseteq X$. When $F$ is finite, then $Y^F/_\sim$ is finite and discrete, because $\fH$ is oligomorphic. Hence, the space 
$$
\prod_{F\in [X]^{<\omega}} Y^F/_\sim
$$
is compact. The mapping $\xi$ from $Y^X/_\sim$ into this space defined by 
$$
[g]_\sim\mapsto \big ([g\rest_F]_\sim\;|\; F\in [X]^{<\omega} \big)
$$ is well-defined. In fact, $\xi$ is a homeomorphism onto a closed subspace thereof. To see this, note that injectivity follows from the definition of the equivalence relation $\sim$, and likewise continuity, since the topology on $Y^X/_\sim$ is precisely given by the behavior of functions on finite sets, modulo the equivalence $\sim$. The fact that the image of $\xi$ is closed follows from the fact that $X$ is countable: when we have, in the range of $\xi$, tuples $([g_i\rest_F]_\sim\;|\; F\in [X]^{<\omega} \big)$ for each $i\in\omega$, and the sequence of these tuples converges in $\prod_{F\in [X]^{<\omega}} Y^F/_\sim$, then a function $g\in Y^X$ such that $([g\rest_F]_\sim\;|\; F\in [X]^{<\omega} \big)$ is the limit of the sequence can be constructed by a standard argument using K\Hung{o}nig's tree lemma. Openness of the mapping $\xi$ is then also obvious. It follows that $Y^X/_\sim$ is indeed a compact Hausdorff space.
\end{proof}

We remark that when $\fH$ is the automorphism group of an $\omega$-categorical first-order structure on $Y$, then the space $Y^X/_\sim$ in Lemma~\ref{lem:space} is nothing but the type space for the theory of that structure with variables indexed by the set $X$. Let us also mention that the condition of $X$ being countable is necessary; cf.~Examples 4.5 and~4.7 in~\cite{schneider}.

\begin{proof}[Proof of Lemma~\ref{lem:lift}]
First assume that $I$ is finite, and write $I=\{0,\ldots, n-1\}$. For each $0\leq i\leq n-1$, we have $f_i\in\overline{\fH\, g_i}$; since $X_i$ is countable, there is a sequence $(\beta_i^{j}\, g_i)_{j\in\omega}$  converging to $f_i$. Now consider the set
%$$
%S:= \{(\id ,\; \beta_0^j,\ldots,\beta_{n-1}^j )\mid j \in {\omega},\; \delta \in \fH\}\; ,
%$$
$$
S:= \{(\id ,\; \beta_0^j,\ldots,\beta_{n-1}^j )\mid j \in {\omega}\}\; ,
$$
viewed as a subset of the space $\overline{\fH}^{n+1}$, with  $\id$ denoting the identity function in $\fH$.
%, and observe that it is invariant under the componentwise action of $\fH$ from the left. 
The space $\overline{\fH}^{n+1}$ can be viewed naturally as a closed subspace of ${(Y^{n+1})}^{(Y^{n+1})}$, and the equivalence relation $\sim$ induced on the latter by the componentwise, oligomorphic action of $\fH$ on $Y^{n+1}$ restricts to $\overline{\fH}^{n+1}$ since this space is invariant under that action. Factoring $\overline{\fH}^{n+1}$ by $\sim$, we obtain a compact space by Lemma~\ref{lem:space}. The equivalence classes of the elements of $S$ have an accumulation point in $(\overline{\fH}^{n+1})_\sim$, which we write as $[(e,e_0,\ldots,e_{n-1})]_\sim$, for some $e,e_0,\ldots,e_{n-1}\in\overline{\fH}$. Hence, there exist $\delta^{j}\in \fH$, for ${j\in\omega}$, such that $(\delta^{j},\delta^{j} \beta_0^{j},\ldots,\delta^{j} \beta_{n-1}^{j})$ converges to $(e,e_0,\ldots,e_{n-1})$. Since for every $0\leq i\leq {n-1}$ we have that $(\beta_i^{j}\, g_i)_{j\in\omega}$ converges to $f_i$, we obtain that $(\delta^j \beta_i^{j}\, g_i)_{j\in\omega}$ converges to $e f_i$; on the other hand, it converges to $e_i g_i$, proving $e f_i= e_i g_i$.

Now assume that $I$ is countably infinite, and assume $I=\omega$. 
By the above, we obtain for every $n\geq 1$ elements $e^n,e_0^n,\ldots,e_{n-1}^n\in\fH$ such that $e^n f_i=e^n_i g_i$ for all $0\leq i \leq n-1$. We can embed the sequences $(e^n,e_0^n,\ldots,e_{n-1}^n)\in\overline{\fH}^{n+1}$ into the product space
$$
\prod_{n\geq 1} \overline{\fH}^{n+1}
$$
by first expanding them to a sequence in $\fH^\omega$ by adding, an infinite number of times, the identity function $\id\in\fH$, and then via the identification of $\fH^\omega$ with a closed subspace of above product space, as in Lemma~\ref{lem:space}. Factoring every component $\overline{\fH}^{n+1}$ of the latter by the equivalence relation $\sim$ induced by the action of $\fH$ on the left, we obtain a compact space. There the equivalence classes of the sequences $(e^n,e_0^n,\ldots,e_{n-1}^n)$ have an accumulation point, namely the equivalence class induced by a sequence $(e,e_0,\ldots)\in\fH^\omega$. Similarly as in the case where $I$ was finite, we conclude $ef_i=e_ig_i$ for all $i\in\omega$.
\end{proof}

The implication from (2) to (3) in Proposition~\ref{prop:canonical} now is a direct consequence of Lemma~\ref{lem:lift}.

\section{Canonisation}
\label{sect:canonisation}
The following is the \emph{canonisation theorem}, first proved combinatorially in~\cite{BPT-decidability-of-definability} in a slightly more specialized context. 

\begin{thm}\label{thm:canonisation}
Let $\G\acts X$, $\fH\acts Y$ be permutation groups, where $X$ is countable, $\G$ is extremely amenable, and $\fH$ is oligomorphic. Let $f\colon X\To Y$. Then
$$
\overline{\fH\, f\, \G}:= \overline{\{\beta\, f\, \alpha\;|\; \alpha\in \G, \beta\in\fH\}}
$$
contains a canonical function with respect to $\G$ and $\fH$.
\end{thm}
\begin{proof}
The space $\overline{\fH\, f\, \G}/_\sim$ is a closed subspace of
the compact Hausdorff space
$Y^X/_\sim$ from Lemma~\ref{lem:space}, and hence
is a compact Hausdorff space as well. We define a continuous action of $\G$ on this space by
$$
(\alpha, [g]_\sim)\mapsto [g\,\alpha^{-1}]_\sim\; .
$$
Clearly, this assignment is a function, it is a group action, and it is continuous. Since $\G$ is extremely amenable, the action has a fixed point $[g]_\sim$. Any member $g$ of this fixed point is canonical: whenever $\alpha\in\G$, then $[g\,\alpha]_\sim=[g]_\sim$, which is the definition of canonicity.
\end{proof}

In applications of Theorem~\ref{thm:canonisation} (e.g., in~\cite{agarwal, PongraczHensonReducts,Poset-Reducts,42,BodJonsPham,AgarwalKompatscher,LinmanPinsker,RandomMinOps,BMPP16,equiv-csps,BodPin-Schaefer-both,Phylo-Complexity,KompatscherPham,BPT-decidability-of-definability,Bodirsky-Mottet,definable-homomorphisms,BP-reductsRamsey}), one usually needs the following special case of the above situation. It states, roughly, that whenever we have a finite arity function $f$ on a countable set, and an oligomorphic extremely amenable permutation group $\G$ on the same set, then we can obtain from $f$ and $\G$, using composition and topological closure,  a canonical function whilst retaining finite information about $f$.

In the following statement, for $m\geq 1$ we write $\fG^m$ for the natural action of $\fG^m$ on $X^m$ given by $((\alpha_1,\ldots,\alpha_m), (x_1,\ldots,x_m))\mapsto (\alpha_1(x_1),\ldots,\alpha_m(x_m))$. Moreover, we denote the pointwise stabilizer of $c^1,\ldots,c^n\in X^m$  in $\fG^m$ by $(\fG^m,c^1,\ldots,c^n)$.

\begin{cor}\label{cor:constants}
Let $\G\acts X$ be an oligomorphic extremely amenable permutation group acting on a countable set $X$. Let $f\colon X^m\To X$ for some $m\geq 1$, and let $c^1,\ldots,c^n\in X^m$ for some $n\geq 1$. Then there exists
$$g\in \overline{\G\, f\, \G^m}$$
such that
\begin{itemize}
\item $g$ agrees with $f$ on $\{c^1,\ldots,c^n\}$, and
\item $g$ is canonical with respect to the groups $(\G^m,c^1,\ldots,c^n)$ and $\G$.
\end{itemize}
\end{cor}
\begin{proof}
The group $\G^m$ is obviously extremely amenable. Moreover, it is known that so is any stabilizer of it (in fact, every open subgroup; cf.~\cite{BPT-decidability-of-definability}). The statement therefore follows from Theorem~\ref{thm:canonisation}.
\end{proof}

\section{An Open Problem}
Is there a converse of Theorem~\ref{thm:canonisation} in the sense that extreme
amenability of $\fG$ is equivalent to some form of the statement of 
the canonisation theorem? 
More precisely, we ask the following question. 

\begin{quest}
Let $\fG\acts X$ be a closed permutation group on a countable domain $X$. 
%, and let $\fH$ be oligomorphic. 
Is it true that $\fG$ is extremely amenable if and only if it has the \emph{canonisation property} of Theorem~\ref{thm:canonisation}, i.e., for every 
oligomorphic permutation group $\fH \acts Y$ and every $f \colon X \to Y$ 
the set
$\overline{\fH \, f \, \fG}$
contains a function that is canonical with respect to $\fG$ and $\fH$?
\end{quest} 
We remark that the canonisation property above implies, for example, that $\fG$ preserves a linear order, as is the case when $\fG$ is extremely amenable. For when $\fH\acts X$ is any oligomorphic extremely amenable permutation group, and $g\in \overline{\fH\id\fG}$ is canonical, then it is easy to see that the preimage under $g$ of any linear order preserved by $\fH$ must be preserved by $\fG$.

After publication of a draft of the present article, Trung Van Pham provided a positive answer to the above question for the case that $\fG$ has an extremely amenable oligomorphic subgroup $\fH$. This is an important case, since the first example of an oligomorphic group $\fG$ not satisfying this condition was discovered only recently~\cite{EvansHubickaNesetril}. Pham's argument is combinatorial, using the Ramsey property; the following proof in the language of groups is the result of discussions with Antoine Mottet and Jakub Opr\v{s}al.

Assuming that $\fG$ is not extremely amenable, we show that $\overline{\fH \id\fG}=\overline{\fG}$ does not contain any canonical function with respect to $\fG$ and $\fH$. To this end, let $S$ be a compact Hausdorff space such that $\fG$ acts continuously on $S$ without fixed point. Since $\fH$ is extremely amenable, the restriction of this action $\fG\acts S$ to $\fH$ does have a fixed point $s\in S$. By restricting $\fG\acts S$ to the closure of the orbit of $s$ in $S$, we may assume that the orbit of $s$ is dense in $S$.

As in the proof of Theorem~\ref{thm:canonisation}, let $\fG$ now act on $\overline{\fG}/_\sim$ by $(\alpha, [g]_\sim)\mapsto [g\,\alpha^{-1}]_\sim$. 
Then the action $\fG\acts S$ is a factor of the action $\fG\acts \overline{\fG}/_\sim$ via the mapping $\phi\colon \overline{\fG}/_\sim \To S$ which sends every $[g]_\sim$ to the limit of $(\alpha_n^{-1}(s))_{n\in\omega}$, for any sequence $(\alpha_n)_{n\in\omega}$ converging to $g$: it is well-defined since $\fH$ fixes $s$, and if $(\beta_n)_{n\in\omega}$ is another such sequence, then $(\beta_n^{-1}\alpha_n)_{n\in\omega}$ converges to the identity, which fixes $s$, and so $(\alpha_n^{-1}(s))_{n\in\omega}$ converges to the limit of $(\beta_n^{-1}(s))_{n\in\omega}$ by continuity. Moreover, by definition $\phi$ is compatible with the two actions, i.e., $\phi([g\,\alpha^{-1}]_\sim)=\alpha(\phi([g]_\sim)$ for all $g\in\overline{\fG}$ and all $\alpha\in\fG$. 

Since $\fG\acts S$ does not have a fixed point, and since it is a factor of $\fG\acts \overline{\fG}/_\sim$, the latter cannot have a fixed point either. As in the proof of Theorem~\ref{thm:canonisation}, fixed points of $\fG\acts \overline{\fG}/_\sim$ correspond precisely to canonical functions with respect to $\fG$ and $\fH$ in $\overline{\fG}$, and we conclude that $\overline{\fG}$ does not contain any canonical function.

\bibliographystyle{plain}
\bibliography{global}

\def\cprime{$'$} \def\cprime{$'$} \def\cprime{$'$}
\begin{thebibliography}{10}

\bibitem{agarwal}
Lovkush Agarwal.
\newblock The reducts of the generic digraph.
\newblock {\em Annals of Pure and Applied Logic}, 167:370--391, 2016.

\bibitem{AgarwalKompatscher}
Lovkush Agarwal and Michael Kompatscher.
\newblock $2^{\aleph_0}$ pairwise nonisomorphic maximal-closed subgroups of
  {S}ym($\mathbb{N}$) via the classification of the reducts of the {H}enson
  digraphs.
\newblock {\em Journal of Symbolic Logic}, 83(2):395--415, 2018.

\bibitem{BKOPP}
Libor Barto, Michael Kompatscher, Miroslav Ol\v{s}\'{a}k, Trung~Van Pham, and
  Michael Pinsker.
\newblock The equivalence of two dichotomy conjectures for infinite domain
  constraint satisfaction problems.
\newblock In {\em Proceedings of the 32nd Annual {ACM/IEEE} Symposium on Logic
  in Computer Science (LICS) 2017}, pages 1--12, 2017.

\bibitem{BKOPP-equations}
Libor Barto, Michael Kompatscher, Miroslav Ol\v{s}\'{a}k, Trung~Van Pham, and
  Michael Pinsker.
\newblock Equations in oligomorphic clones and the constraint satisfaction
  problem for $\omega$-categorical structures.
\newblock {\em Journal of Mathematical Logic}, 19(2):\#1950010, 2019.

\bibitem{wonderland}
Libor Barto, Jakub Opr\v{s}al, and Michael Pinsker.
\newblock The wonderland of reflections.
\newblock {\em Israel Journal of Mathematics}, 223(1):363--398, 2018.

\bibitem{BartoPinskerDichotomy}
Libor Barto and Michael Pinsker.
\newblock The algebraic dichotomy conjecture for infinite domain constraint
  satisfaction problems.
\newblock In {\em Proceedings of the 31th {A}nnual {IEEE} {S}ymposium on
  {L}ogic in {C}omputer Science (LICS) 2016}, pages 615--622, 2016.

\bibitem{Topo}
Libor Barto and Michael Pinsker.
\newblock Topology is irrelevant.
\newblock {\em SIAM Journal on Computing}, 49(2):365--393, 2020.

\bibitem{bergelson}
Vitaly Bergelson.
\newblock Ergodic {R}amsey theory: a dynamical approach to static theorems.
\newblock In {\em Proceedings of the International Congress of Mathematicians},
  volume~II, pages 1655--1678, Z\"{u}rich, 2006. European Mathematical Society.

\bibitem{BodJonsPham}
Manuel Bodirsky, Peter Jonsson, and Trung~Van Pham.
\newblock The reducts of the homogeneous binary branching {C}-relation.
\newblock {\em Journal of Symbolic Logic}, 81(4):1255--1297, 2016.

\bibitem{Phylo-Complexity}
Manuel Bodirsky, Peter Jonsson, and Trung~Van Pham.
\newblock {The Complexity of Phylogeny Constraint Satisfaction Problems}.
\newblock {\em ACM Transactions on Computational Logic (TOCL)}, 18(3), 2017.
\newblock An extended abstract appeared in the conference STACS 2016.

\bibitem{BodJunker}
Manuel Bodirsky and Markus Junker.
\newblock $\aleph_0$-categorical structures: interpretations and endomorphisms.
\newblock {\em Algebra Universalis}, 64(3-4):403--417, 2011.

\bibitem{BMPP16}
Manuel Bodirsky, Barnaby Martin, Michael Pinsker, and Andr{\'{a}}s
  Pongr{\'{a}}cz.
\newblock Constraint satisfaction problems for reducts of homogeneous graphs.
\newblock {\em SIAM Journal on Computing}, 48(4):1224--1264, 2019.
\newblock A conference version appeared in the Proceedings of the 43rd
  International Colloquium on Automata, Languages, and Programming ({ICALP})
  2016, pages 119:1--119:14.

\bibitem{Bodirsky-Mottet}
Manuel Bodirsky and Antoine Mottet.
\newblock Reducts of finitely bounded homogeneous structures, and lifting
  tractability from finite-domain constraint satisfaction.
\newblock In {\em Proceedings of the 31th Annual IEEE Symposium on Logic in
  Computer Science (LICS)}, pages 623--632, 2016.

\bibitem{BP-reductsRamsey}
Manuel Bodirsky and Michael Pinsker.
\newblock Reducts of {R}amsey structures.
\newblock {\em AMS Contemporary Mathematics, vol. 558 (Model Theoretic Methods
  in Finite Combinatorics)}, pages 489--519, 2011.

\bibitem{RandomMinOps}
Manuel Bodirsky and Michael Pinsker.
\newblock Minimal functions on the random graph.
\newblock {\em Israel Journal of Mathematics}, 200(1):251--296, 2014.

\bibitem{BodPin-Schaefer-both}
Manuel Bodirsky and Michael Pinsker.
\newblock Schaefer's theorem for graphs.
\newblock {\em Journal of the ACM}, 62(3):52 pages (article number 19), 2015.
\newblock A conference version appeared in the Proceedings of the 43rd Annual
  ACM Symposium on Theory of Computing (STOC) 2011, pages 655--664.

\bibitem{Topo-Birk}
Manuel Bodirsky and Michael Pinsker.
\newblock Topological {B}irkhoff.
\newblock {\em Transactions of the American Mathematical Society},
  367:2527--2549, 2015.

\bibitem{42}
Manuel Bodirsky, Michael Pinsker, and Andr\'{a}s Pongr\'acz.
\newblock The 42 reducts of the random ordered graph.
\newblock {\em Proceedings of the London Mathematical Society},
  111(3):591--632, 2015.

\bibitem{BPP-projective-homomorphisms}
Manuel Bodirsky, Michael Pinsker, and Andr\'{a}s Pongr\'acz.
\newblock Projective clone homomorphisms.
\newblock {\em Journal of Symbolic Logic}, 2019.
\newblock To appear. doi:10.1017/jsl.2019.23.

\bibitem{BPT-decidability-of-definability}
Manuel Bodirsky, Michael Pinsker, and Todor Tsankov.
\newblock Decidability of definability.
\newblock {\em Journal of Symbolic Logic}, 78(4):1036--1054, 2013.
\newblock A conference version appeared in the Proceedings of the 26th Annual
  IEEE Symposium on Logic in Computer Science ({LICS}) 2011, pages 321-328.

\bibitem{equiv-csps}
Manuel Bodirsky and Micha\l\ Wrona.
\newblock Equivalence constraint satisfaction problems.
\newblock In {\em Proceedings of Computer Science Logic}, volume~16 of {\em
  LIPICS}, pages 122--136. Dagstuhl Publishing, September 2012.

\bibitem{EvansHubickaNesetril}
David~M. Evans, Jan Hubi\v{c}ka, and Jaroslav Ne\v{s}et\v{r}il.
\newblock Automorphism groups and {Ramsey} properties of sparse graphs.
\newblock {\em Proceedings of the London Mathematical Society}, 119:515--546,
  2019.

\bibitem{Hodges}
Wilfrid Hodges.
\newblock {\em A shorter model theory}.
\newblock Cambridge University Press, Cambridge, 1997.

\bibitem{Topo-Dynamics}
Alexander Kechris, Vladimir Pestov, and Stevo Todor\v{c}evi\'c.
\newblock Fra\"{i}ss\'e limits, {R}amsey theory, and topological dynamics of
  automorphism groups.
\newblock {\em Geometric and Functional Analysis}, 15(1):106--189, 2005.

\bibitem{Kechris-nonarchimedian}
Alexander~S. Kechris.
\newblock Dynamics of non-archimedean {P}olish groups.
\newblock In {\em Proceedings of the {E}uropean Congress of Mathematics,
  Krakow}, pages 375--397. European Math. Society, 2014.

\bibitem{locally-finite}
Bartek Klin, Eryk Kopczynski, Joanna Ochremiak, and Szymon Toru{\'{n}}czyk.
\newblock Locally finite constraint satisfaction problems.
\newblock In {\em 30th Annual {ACM/IEEE} Symposium on Logic in Computer
  Science, {LICS} 2015, Kyoto, Japan}, pages 475--486, 2015.

\bibitem{definable-homomorphisms}
Bartek Klin, Slawomir Lasota, Joanna Ochremiak, and Szymon Toru\'{n}czyk.
\newblock {Homomorphism Problems for First-Order Definable Structures}.
\newblock In Akash Lal, S.~Akshay, Saket Saurabh, and Sandeep Sen, editors,
  {\em 36th IARCS Annual Conference on Foundations of Software Technology and
  Theoretical Computer Science (FSTTCS 2016)}, volume~65 of {\em Leibniz
  International Proceedings in Informatics (LIPIcs)}, pages 14:1--14:15,
  Dagstuhl, Germany, 2016. Schloss Dagstuhl--Leibniz-Zentrum fuer Informatik.

\bibitem{KompatscherPham}
Michael Kompatscher and Trung~Van Pham.
\newblock A complexity dichotomy for poset constraint satisfaction.
\newblock {\em IfCoLog Journal of Logics and their Applications ({FLAP})},
  5(8):1663--1696, 2018.
\newblock An extended abstract appeared at the 34th Symposium on Theoretical
  Aspects of Computer Science (STACS) 2017.

\bibitem{LinmanPinsker}
Julie Linman and Michael Pinsker.
\newblock Permutations on the random permutation.
\newblock {\em Electronic Journal of Combinatorics}, 22(2):1--22, 2015.

\bibitem{MottetPinskerCores}
Antoine Mottet and Michael Pinsker.
\newblock Cores over {Ramsey} structures.
\newblock Preprint arXiv:2004.05936, 2020.

\bibitem{Poset-Reducts}
P\'{e}ter~P\'{a}l Pach, Michael Pinsker, Gabriella Pluh\'{a}r, Andr\'{a}s
  Pongr\'{a}cz, and Csaba Szab\'{o}.
\newblock Reducts of the random partial order.
\newblock {\em Advances in Mathematics}, 267:94--120, 2014.

\bibitem{PongraczHensonReducts}
Andr\'as Pongr\'{a}cz.
\newblock Reducts of the {H}enson graphs with a constant.
\newblock {\em Annals of Pure and Applied Logic}, 168(7):1472--1489, 2017.

\bibitem{schneider}
Friedrich~Martin Schneider.
\newblock A uniform {B}irkhoff theorem.
\newblock {\em Algebra Universalis}, 78(3):337--354, 2017.

\end{thebibliography}

\end{document}